\documentclass[article]{amsart}

\usepackage{amssymb,amsfonts,amsmath,amsthm}
\usepackage[all,arc]{xy}
\usepackage{enumerate}
\usepackage{mathrsfs}
\usepackage[toc,page]{appendix}
\usepackage[left=3cm, right=3cm, bottom=3cm]{geometry}
\usepackage{tabularx}
\usepackage{url}
\usepackage{color}
\usepackage{comment}
\usepackage{hyperref}
\usepackage{tikz-cd}  
\usepackage{mathdots} 
\usepackage{romannum} 

\newtheorem{thm}{Theorem}[section]

\newtheorem*{thm*}{Theorem}
\newtheorem*{cor*}{Corollary}
\newtheorem*{prop*}{Proposition}
\newtheorem{cor}[thm]{Corollary}

\newtheorem{lem}[thm]{Lemma}

\newtheorem{rmk}[thm]{Remark}

\theoremstyle{definition}
\newtheorem{defn}[thm]{Definition}

\newtheorem{exmp}[thm]{Example}

\newtheorem*{notn*}{Notation}

\theoremstyle{remark}
\newtheorem{rem}[thm]{Remark}

\newtheorem*{idea*}{Idea}

\makeatletter
\let\c@equation\c@thm
\makeatother
\numberwithin{thm}{section}
\numberwithin{equation}{section}

\title[Root Multiplicities of Rank 3 Kac-moody Algebras]{A Combinatorial Approach to Root Multiplicities of a Special Type Rank 3 Kac-moody Algebras}

\author{Bowen Chen, Hanyi Luo and  Hao Sun}

\begin{document}
\pagenumbering{arabic}
\maketitle
\begin{abstract}
In this paper, we calculate the dimension of root spaces $\mathfrak{g}_{\lambda}$ of a special type rank $3$ Kac-Moody algebras $\mathfrak{g}$. We first introduce a special type of elements in $\mathfrak{g}$, which we call elements in standard form. Then, we prove that any root space is spanned by these elements. By calculating the number of linearly independent elements in standard form, we obtain a formula for the dimension of root spaces $\mathfrak{g}_{\lambda}$, which depends on the root $\lambda$.
\end{abstract}

\flushbottom



\renewcommand{\thefootnote}{\fnsymbol{footnote}}
\footnotetext[1]{MSC2010 Class: 17B67, 17B22}

\section{Introduction}\label{sec:1}
Kac-Moody algebras are generalizations of finite-dimensional semisimple Lie algebras, of which the structures can be defined from generalized Cartan matrices. The dimensions of root spaces (also called \emph{root multiplicities}) are important data in understanding the structure of Kac-Moody algebras. Generally speaking, Kac-Moody algebras are divided into three groups: finite case, affine case and indefinite case (see \cite{Kac}). Finite and affine Kac-Moody algebras were studied in many different ways in the last decades, while the indefinite type is still mysterious as its name implies. As a special case of indefinite Kac-Moody algebras, hyperbolic Kac-Moody algebras have some interesting properties, and people have made a full classification and studied the root multiplicities in the last several decades \cite{CarChu,Fren,Kang}. In this paper, we study the dimension of root spaces of a special type indefinite Kac-Moody algebras with rank $3$.

Let $A$ be a generalized Cartan matrix, and denote by $\mathfrak{g}(A)$ the corresponding Kac-Moody algebra. Let $A$ be a matrix in the following type
\begin{align*}
\begin{pmatrix}
2 &-a_1 & 0 \\
-a_1 &2 & -a_2\\
0 & -a_2 & 2\\
\end{pmatrix},
\end{align*}
where $a_1,a_2$ are positive integers and at least one of them is larger than one. We will calculate the dimension of root spaces of the corresponding Kac-Moody algebra $\mathfrak{g}:=\mathfrak{g}(A)$ in this paper. The Kac-Moody algebra $\mathfrak{g}$ has three simple roots $\alpha_1,\alpha_2,\alpha_3$. Denote by $e_i$ the element spanning the root space $\mathfrak{g}_{\alpha_i}$. From this special matrix $A$, we also know that $[e_1,e_3]=0$, which is also an important property.

In \S \ref{sec:2}, we first consider a special type of elements in $\mathfrak{g}$, which are called elements in standard form. These are elements that can be written in the following way
\begin{align*}
    [e_{a_{n}},[e_{a_{n-1}},[\dots,[e_{a_{2}},e_{a_{1}}]\dots]]].
\end{align*}
Let $\lambda=n_1 \alpha_1 +n_2 \alpha_2 + n_3 \alpha_3$ be a positive root, and $n=n_1+n_2+n_3$. We prove that any root space $\mathfrak{g}_{\lambda}$ can be spanned by elements in standard form (see Theorem \ref{thm:2.3} and Corollary \ref{cor:2.4}). Note that given an $n$-tuple $(a_n,a_{n-1},\dots,a_1)$, we can equate it with an element in standard form $[e_{a_{n}},[e_{a_{n-1}},[\dots,[e_{a_{2}},e_{a_{1}}]\dots]]]$. Therefore, we calculate the dimension of root spaces $\mathfrak{g}_{\lambda}$ by calculating the corresponding number of $n$-tuples. With respect to the basic property of Lie algebras (see \cite{Hum})
\begin{align*}
    [\mathfrak{g}_{\alpha} ,\mathfrak{g}_\beta] \subseteq \mathfrak{g}_{\alpha+\beta},
\end{align*}
where $\alpha,\beta$ are roots, we know that $\mathfrak{g}_{\lambda}$ is spanned by elements in standard form $$[e_{a_{n}},[e_{a_{n-1}},[\dots,[e_{a_{2}},e_{a_{1}}]\dots]]]$$ such that the number of $e_i$ is $n_i$. Although an $n$-tuple corresponds to an element in standard form, this element can be trivial, and several $n$-tuples may correspond to linearly dependent elements. Thus, we have to eliminate all of these bad cases. 

In \S \ref{sec:3}, we introduce a combinatorial problem, which will be used in calculating the dimension. In \S \ref{sec:4}, we make a careful discussion of the $n$-tuples together with the corresponding elements in $\mathfrak{g}$. Here are several important properties.
\begin{enumerate}
    \item If $e_1$ and $e_3$ are adjacent in a standard form $g$, then we will get a new element $g'$ in standard form by switching the position of $e_1$ and $e_3$. We find that $g'=-g$ (see Lemma \ref{lem:4.3}).
    \item If $e_i$ and $e_2$ are adjacent in a standard form $g$, where $i=1,3$, switching the position will give a new element $g'$, which is linearly independent to $g$ (see Lemma \ref{lem:4.5}).
\end{enumerate}
The properties above reduce the problem of calculating the dimension of $\mathfrak{g}_{\lambda}$ to a combinatorial problem such that we have $n_1$ red balls, $n_3$ blue balls, and we want to put them into $n_2$ boxes (under some extra conditions in Lemma \ref{lem:4.5}). With a careful discussion of the cases in the first box (see Lemma \ref{lem:4.8}), we have the following upper bound of $\dim \mathfrak{g}_{\lambda}$
\begin{align*}
    A:=\sum_{0\leq i \leq a_1, 0\leq j \leq a_2 \atop (i,j)\neq (0,0)}{n_1+n_2-i-2 \choose n_2-2}{n_2+n_3-j-2 \choose n_2-2}.
\end{align*}
However, there are some $n$-tuples left corresponding to trivial elements in $\mathfrak{g}_\lambda$. With a careful discussion in \S \ref{sec:4.3.1}, we have the following number of ``trivial" $n$-tuples
\begin{align*}
    C=C_1+C_2:={n_1+n_2+n_3-a_1-2 \choose n_2-a_1-1}+{n_1+n_2+n_3-a_2-2 \choose n_2-a_2-1}.
\end{align*}
After deleting all ``trivial" elements, there are still some $n$-tuples linearly dependent to each other. We discuss this case in \S \ref{sec:4.3.2}, and the number is given as follows
\begin{align*}
    B:={n_2+n_3-4 \choose n_2-2} {n_1+n_2-2 \choose n_2-2}+{n_1+n_2-4 \choose n_2-2} {n_2+n_3-2 \choose n_2-1}.
\end{align*}

With respect to the discussion above, we have our main theorem.
\begin{thm}[Theorem \ref{thm:4.11}]
Suppose that $a_1 \leq a_2$ and $n_1,n_2,n_3 \geq 2$. Let $\lambda = \sum_{i=1}^{3}n_i \alpha_i$. Denote by $\mathfrak{g}_{\lambda}$ the corresponding root space. We have
\begin{align*}
\dim \mathfrak{g}_{\lambda} = \left\{\begin{matrix}
A-B, &n_2 < \min\left \{1+a_1,1+a_2  \right \} \\ 
A-B-C_1, & 1+a_1 \leq n_2 < 1+a_2\\ 
A-B-C_1-C_2, & n_2 \geq \max\left \{1+a_1,1+a_2  \right \}
\end{matrix}\right..
\end{align*}
\end{thm}

Note that in the theorem, we assume that $n_i \geq 2$ for $i=1,2,3$. In Remark \ref{rmk:4.11.1}, we briefly discuss the general cases that $n_i \geq 0$.

As an application of the main theorem, we calculate the dimension of root spaces of a special $3 \times 3$ hyperbolic Kac-Moody algebra (see Corollary \ref{cor:4.13}).

\section{Standard Form}\label{sec:2}
We first give the definition of Kac-Moody algebras. We refer the readers to \cite{Kac} for more details. Let $A$ be a $r \times r$ generalized Cartan matrix (GCM) of rank $l$. A \emph{realization} of $A$ is a triple $(\mathfrak{h},\Pi,\Pi^{\vee})$, where $\mathfrak{h}$ is a vector space, $\Pi=\{\alpha_1,\dots,\alpha_r\} \subseteq \mathfrak{h}^*$ and $\Pi^{\vee}=\{\alpha_1^{\vee} ,\dots, \alpha_r^{\vee} \} \subseteq \mathfrak{h}$, such that
\begin{enumerate}
\item both sets $\Pi$ and $\Pi^{\vee}$ are linearly independent.
\item $(\alpha_i^{\vee},\alpha_j)=a_{ij}$, $1 \leq i,j \leq r$,
\item $r-l=\dim \mathfrak{h}-r$.
\end{enumerate}
We fix a generalized Cartan matrix $A$. Let $(\mathfrak{h},\Pi,\Pi^{\vee})$ be a realization of it. The \emph{Kac-Moody algebra} $\mathfrak{g}:=\mathfrak{g}(A)$ is a Lie algebra (with Lie bracket $[,]$) generated by $\mathfrak{h}$ and $e_i,f_i$, where $1 \leq i \leq r$, such that
\begin{enumerate}
\item $[e_i,f_j]=\delta_{ij} \alpha_i^{\vee}$, $1 \leq i,j \leq r$.
\item $[h,h']=0$, where $h,h' \in \mathfrak{h}$.
\item $[h,e_i]=(\alpha_i,h)e_i$.
\item $[h,f_i]=-(\alpha_i,h)f_i$.
\item If $i \neq j$, we have ${\rm ad}(e_i)^{1- a_{ij}}(e_j)=0$ and ${\rm ad}(f_i)^{1- a_{ij}}(f_j)=0$, where ${\rm ad}(x)(y):=[x,y]$ is the adjoint representation of $\mathfrak{g}$ with respect to the Lie algebra structure.
\end{enumerate}
The elements $\alpha_i$ are called \emph{simple roots}. Let $\lambda=\sum\limits_{i=1}^r n_i \alpha_i$ be a root of $\mathfrak{g}$, and denote by $\mathfrak{g}_{\lambda}$ the root space of $\lambda$. In this section, we study some special elements in $\mathfrak{g}$, which we call \emph{elements in standard form}. We prove that the root space $\mathfrak{g}_{\lambda}$ has a basis, of which the elements are in standard form. For simplicity, we only work on positive roots and the corresponding root space $\mathfrak{g}_{\lambda}$ is generated by $\{e_i\}_{1 \leq i \leq r}$.


\begin{defn}\label{Def:2.1}
An element $g \in \mathfrak{g}$ is in \emph{standard form of length $n$} if it can be written in the following way
\begin{align*}
    [e_{a_{n}},[e_{a_{n-1}},[\dots[e_{a_{2}},e_{a_{1}}]\dots]]]
    ,
\end{align*}
where $1\leq a_i\leq r$ for any integer $r>1$.
\end{defn}

Following the definitions above, $[e_1,[e_2,e_3]]$ is a standard form of length $3$. $[[e_1,e_2],[e_3,e_4]]$ is an element of length $4$ but not in standard form.

\begin{lem}\label{lem:2.2}
Given two elements $X, Y \in \mathfrak{g}$ in standard form of length $i$ and $j$ respectively, $[X,Y]$ can be written as a linear sum of elements in standard form of length $i+j$.
\end{lem}

\begin{proof}
Denote by $k_i$ an element of length $i$ in standard form. Without loss of generality, let $X=k_i$ and $Y=k_j$.

We first start from the base case $[k_{2},k_{n}]$. Let $k_2$ be $[e_{a_{1}},e_{a_{2}}]$, then we have
\begin{align*}
[k_2,k_n]=&-[k_n,[e_{a_{1}},e_{a_{2}}]].\\
\end{align*}
Based on Jacobi Identity, one can write
\begin{align*}
=&[e_{a_{1}},[e_{a_{2}},k_{n}]]+[e_{a_{2}},[k_{n},e_{a_{1}}]]\\
=&[e_{a_{1}},[e_{a_{2}},k_{n}]]-[e_{a_{2}},[e_{a_1},k_n]].
\end{align*}
Since $k_n$ is in standard form, the two elements above $[e_{a_{1}},[e_{a_{2}},k_{n}]]$ and $[e_{a_{2}},[e_{a_1},k_n]]$ are also in standard form of length $n+2$.

Then, we keep computing $[k_3,k_n]$ to find the pattern. Here, we let $k_3=[e_{a_{1}},[e_{a_{2}},e_{a_{3}}]]$.
\begin{align*}
[k_3,k_n]
=&-[k_n,[e_{a_{1}},[e_{a_{2}},e_{a_{3}}]]]\\
=&[e_{a_{1}},[[e_{a_{2}},e_{a_{3}}],k_n]]+[[e_{a_{2}},e_{a_{3}}],[k_n,e_{a_{1}}]]\\
=&[e_{a_{1}},([e_{a_{2}},[e_{a_{3}},k_n]-[e_{a_{3}},[e_{a_{2}},k_n])]+[e_{a_{2}},[e_{a_{3}},[e_{a_{1}},k_n]]]-[e_{a_{3}},[e_{a_{2}},[e_{a_{1}},k_n]]]\\
=&[e_{a_{1}},[e_{a_{2}},[e_{a_{3}},k_n]]-[e_{a_{1}},[e_{a_{3}},[e_{a_{2}},k_n]]]+[e_{a_{2}},[e_{a_{3}},[e_{a_{1}},k_n]]]-[e_{a_{3}},[e_{a_{2}},[e_{a_{1}},k_n]]]
\end{align*}

Clearly, four elements in the last equation are in standard form of length $n+3$.

We now conduct induction to prove $[k_m, k_n]$ can be written as the sum of standard forms of length $m+n$. Assume $m < n$ and $k_m=[e_{a_{1}},k_{m-1}]$. Therefore, we can write
\begin{align*}
[k_m,k_n]
=&[[e_{a_{1}},k_{m-1}],k_n]\\
=&[e_{a_{1}},[k_{m-1},k_n]]+[k_{m-1},[k_n,e_{a_{1}}]\\
=&[k_1,[k_{m-1},k_n]]+[k_{m-1},k_{n+1}]
\end{align*}
By induction, $[k_{m-1},k_n]$ can be written as a sum of elements in standard forms of length $m+n-1$, which implies that $[e_{a_{1}},[k_{m-1},k_n]]$ can be written as a sum of elements in standard forms of length $m+n$. The same argument holds for $[k_{m-1},k_{n+1}]$. Therefore, $[k_m,k_n]$ can be written as a sum of elements in standard forms of length $m+n$.

\end{proof}

Lemma \ref{lem:2.2} proves two elements in standard forms in Lie bracket can be written as the linear sum of standard forms. Then, in the theorem below, we are going to prove a more general result.

\begin{thm} \label{thm:2.3}
For any element $X, Y$ of length $m$ and $n$ respectively $(m, n > 0)$, $[X,Y]$ can be written as the sum of standard forms of length $l=m+n$.
\end{thm}

\begin{proof}
We try to prove this inductively on the total length $l$. Note that $n = l-m$ since the total length is $l$. When $l=2$, we have the base case
\begin{align*}
    [X,Y] = [e_{a_{1}}, e_{a_{2}}],
\end{align*}
where $[e_{a_{1}}, e_{a_{2}}]$ is a standard form of length $2$.

Suppose when $l < l'$, $[X,Y]$ can be converted to a linear sum of elements in standard form of length $l$. We want to prove the statement holds when $l=l'$. In this case, $X$ is an element of length $m$ and $Y$ is an element of length $l'-m$. Both have length less than $l'$ which implies they can be converted to a linear sum of standard forms from the induction hypothesis. Because of the bilinear property of lie bracket, without loss of generality, suppose that $X$ and $Y$ are elements in standard forms of length $m$ and $l'-m$ respectively. Based on the results from Lemma \ref{lem:2.2}, then $[X,Y]$ can be written as a linear sum of length $l'$, which proves the theorem.
\end{proof}

With the theorem above, we can write, for example, $[[e_1,e_2],[[e_1,e_3],[e_2,e_3]]]$, which is an element of length $6$, as a linear sum of elements in standard forms of length $6$.
\begin{align*}
    &[[e_1,e_2],[[e_1,e_3],[e_2,e_3]]]\\
   =&[[e_1,e_2],[e_2,[e_3,[e_3,e_1]]]+[[e_1,e_2],[e_3,[e_2,[e_1,e_3]]]]\\
    =&[e_1,[e_2,[e_2,[e_3,[e_3,e_1]]]+[e_2,[e_1,[e_2,[e_3,[e_1,e_3]]]]]\\
    +&[e_1,[e_2,[e_3,[e_2,[e_1,e_3]]]]]+[e_2,[e_1,[e_3,[e_2,[e_3,e_1]]]]].
\end{align*}

As an application of Theorem \ref{thm:2.3}, we have the following corollary.
\begin{cor}\label{cor:2.4}
Let $\lambda=\sum\limits_{i=1}^r n_i \alpha_i$ be a root of $\mathfrak{g}$. The root space $\mathfrak{g}_\lambda$ is spanned by elements in standard form.
\end{cor}

\section{A Combinatorial Problem}\label{sec:3}

In this section, we are introducing a combinatorial problem that is related to our main problem.

\begin{lem}\label{lem:3.1}
Suppose we have $n_1$ non-distinguishable red balls and $n_3$ non-distinguishable blue balls putting into $n_2$ boxes $(n_1, n_2, n_3 \geq 0)$, without exclusion (boxes can be empty). There are 
\begin{align*}
    {n_1+n_2-1 \choose n_2-1}{n_2+n_3-1 \choose n_2-1}
\end{align*}
ways to arrange the balls.
\end{lem}

\begin{proof}
There are ${n+l-1 \choose l-1}$ ways to put $n$ non-distinguishable balls into $l$ boxes without exclusion. We can prove the lemma based on this theorem above.
\begin{enumerate}
    \item  We first consider putting all red balls into the boxes regardless of the blue balls, which has ${n_1+n_2-1 \choose n_2-1}$ ways.
    \item  Similarly, if we put all blue balls into the boxes, there are ${n_2+n_3-1 \choose n_2-1}$ ways.
\end{enumerate}
 Then we combine the situations together since putting red balls in the boxes and putting blue ones in the boxes are two independent events. Hence, the total number of ways to put the balls are
 \begin{align*}
     {n_1+n_2-1 \choose n_2-1}{n_2+n_3-1 \choose n_2-1}.
 \end{align*}
\end{proof}

\begin{exmp}\label{Exp:3.2}
We let $n=2$ and $l=4$ and see how we get ${n+l-1 \choose l-1}$ as the number of ways to put without exclusion $n$ balls into $l$ boxes. We will explain how to derive the number.

Let us first consider a sequence of numbers with the length of $5$, which corresponds to $n+l-1$, and we have $l-1=3$ boards and $n=2$ balls to be put in all 5 positions. Clearly, we have ${5 \choose 2}=10$ ways in total. Put all the boards and balls into a sequence. Note that $l-1$ boards break the sequence into $l$ parts. For example, $(a,b,b,a,b)$ has four parts: with $a$ in first part and third part (parts starting on the left) and nothing in others. Therefore, each part corresponds to a box and the number of balls in each part is the number of balls we will put into the corresponding box.

This sequence corresponds to a possible case, several of which are listed as follow:
\begin{enumerate}
    \item $(a,a,b,b,b)$
    \item $(a,b,a,b,b)$
    \item $(a,b,b,a,b)$
\end{enumerate}
where $a$ represents the ball, and $b$ represents the board.
We have several rules for these balls and boards.
\begin{enumerate}
    \item Each interval between two $b$ is a box, and the bracket and its closet $b$ form a box.
    \item The number of $a$ inside each box is the number of balls that is put into the box.
    \item The boxes can be empty.
    \item The boxes are distinguishable, and the balls are non-distinguishable.
\end{enumerate}
When we look at our own example, we have $3$ boards, which means we have $4$ boxes, including the boxes bounded by the bracket. We have two balls to be put in four boxes. Thus, we transform the permutation question into what we want to solve. Moreover, they are respectively corresponded to the situations, several of which are listed as follow:
\begin{enumerate}
    \item $2$ $0$ $0$ $0$ corresponds to $(a,a,b,b,b)$
    \item $1$ $1$ $0$ $0$ corresponds to $(a,b,a,b,b)$
    \item $1$ $0$ $1$ $0$ corresponds to $(a,b,b,a,b)$
\end{enumerate}
where the first number represents the number of balls in the first interval, and the second in the second box, and so on.

Thus, the number of ways to put $n$ balls into $l$ boxes is equivalent to that to put $l-1$ boards to separate $n$ numbers, which is exactly ${n+l-1 \choose l-1}$.
\end{exmp}

\section{Dimension of Root Spaces}\label{sec:4}

\subsection{Properties of n-tuples}\label{sec:4.1}

Let $A$ be a $3 \times 3$ generalized Cartan matrix in the following type
\begin{align*}
\begin{pmatrix}
2 &-a_1 & 0 \\
-a_1 &2 & -a_2\\
0 & -a_2 & 2\\
\end{pmatrix},
\end{align*}
where $a_1$ and $a_2$ are positive integers and at least one of them is larger than one. The corresponding Kac-Moody algebra is denoted by $A$.

For the special matrix $A$, we can write as follow four specific formulas from property $(5)$ in the definition of Kac-Moody algebra (see \S \ref{sec:2}):
\begin{enumerate}
    \item ${\rm ad}(e_1)^{1+ a_1}(e_2)=0$.
    \item ${\rm ad}(e_2)^{1+ a_1}(e_1)=0$.
    \item ${\rm ad}(e_2)^{1+ a_2}(e_3)=0$.
    \item ${\rm ad}(e_3)^{1+ a_2}(e_2)=0$.
    \item $[e_1,e_3]=[e_3,e_1]=0$.
\end{enumerate}
We will use these properties frequently when proving our main result.

Now we fix a root $\lambda = \sum\limits_{i=1}^{3 }n_i\alpha_i$. Denote by $\mathfrak{g}_{\lambda}$ the root space.  Given an element in standard form
\begin{align*}
    [e_{a_{n}},[e_{a_{n-1}},[\dots,[e_{a_{2}},e_{a_{1}}]\dots]]] \in \mathfrak{g}_{\lambda},
\end{align*}
it can be associated to an $n$-tuple $(a_n,a_{n-1}, \dots,a_1)$ naturally, where
\begin{itemize}
\item $1 \leq a_i \leq 3$ for $1 \leq i \leq n$;
\item for each positive integer $j$, $|\{a_i | a_i = j , 1 \leq i \leq n\}|=n_j$.
\end{itemize}
Conversely, for each $n$-tuple $(a_n,a_{n-1},\dots,a_1)$, the corresponding element in standard form is
\begin{align*}
    [e_{a_{n}},[e_{a_{n-1}},[\dots,[e_{a_{2}},e_{a_{1}}]\dots]]].
\end{align*}
Therefore, any element in standard form corresponds to an $n$-tuple as the discussion above, and we will use the $n$-tuples to represent the elements in standard form. If we want to find the dimension of $\mathfrak{g}_{\lambda}$, we have to find a maximal linearly independent set in $\mathfrak{g}_{\lambda}$. By Theorem \ref{thm:2.3}, we can assume that the elements in this set are in standard form. In this case, the $n$-tuples is more convenient to work with. However, if we work on $n$-tuples, there are some problems, when we do the calculations:
\begin{itemize}
    \item Several distinct $n$-tuples may correspond to linearly dependent elements in $\mathfrak{g}_{\lambda}$.
    \item A non-trivial $n$-tuple may correspond to a trivial element in $\mathfrak{g}_{\lambda}$.
\end{itemize}
In this section, we calculate the dimension of root spaces by solving the problems above.

\begin{defn}\label{def:4.1}
An $n$-tuple is \emph{nontrivial}, if the corresponding element in standard form is nontrivial. A collection of $n$-tuples are \emph{linearly independent} if their corresponding standard forms are linearly independent.
\end{defn}

\begin{lem}\label{lem:4.2}
Let $(a_n, \dots, a_1)$ be a nontrivial $n$-tuple. Then, it is linearly equivalent to an $n$-tuple $(b_n,\dots,b_1)$ such that $b_1=2$.
\end{lem}

\begin{proof}
By property $(5)$, we know that $[e_1,e_3] = 0$. Therefore, one of $a_1$ and $a_2$ should be $``2"$ to make the element non-trivial. Now we assume that $a_2=2$. Since 
\begin{align*}
    [e_2,e_{a_1}]=-[e_{a_1},e_2], 
\end{align*}
the $n$-tuples $(a_n,\dots,a_2,a_1)$ and $(a_n,\dots,a_1,a_2)$ are linearly equivalent. This finishes the proof of this lemma.
\end{proof}

By this lemma, we can always assume that the first element in any nontrivial $n$-tuple is $``2"$.

\begin{lem}\label{lem:4.3}
Let $n \geq 2$ be a positive integer. For any nontrivial $n$-tuple
\begin{align*}
    (a_{n},\dots,p,q,\dots,a_{2},a_{1})
\end{align*} 
such that $a_{i+1}=p$, $a_{i}=q$ and $p,q = 1$ or $3$, this element is linear equivalent to the element 
\begin{align*}
    (a_{n},\dots,q,p,\dots, a_{2},a_{1})
\end{align*}
in standard form of the same length, where we swap the position of $p$ and $q$.
\end{lem}

\begin{proof}
We will first prove elements in the following cases
\begin{align*}
    (p,q,a_{i-1},\dots,a_{2},a_{1}),
\end{align*}
where $a_{i+1}=p$ and $a_{i}=q$. It is easy to check that the argument holds when $i=2,3$. We now focus on elements in standard form of length greater than $4$. We consider the $[e_p,[e_q,A]]$, where $A=[e_{a_{i-1}},[\dots. ,  [e_{a_{2}},e_{a_{1}}]\dots]$ 
is an element in standard form of length $i-1$.

Based on the Jacobi Identity,
\begin{align*}
[e_p,[e_q,A]]
=&-[e_q,[A,e_p]]-[A,[e_p,e_q]]
\end{align*}
As we know $[e_1,e_3]=0$ given by property $(5)$, and $[e_i,e_i]=0$ ($i = 1,3$),  $[e_p,[e_q,A]]
=[e_q,[e_p,A]]$.

The same proof holds for the general case $(a_n,\dots,p,q,\dots,a_1)$.
\end{proof}


\begin{exmp}\label{exp:4.4}
Take $[e_2,[e_1,[e_1,[e_3,e_2]]]$ as an example. Let us find the number of elements that are equal to $[e_2,[e_1,[e_1,[e_3,e_2]]]]$.
By applying Jacobi Identity, we can get
\begin{align*}
&[e_2,[e_1,[e_1,[e_3,e_2]]]]\\
=&[e_2,[e_1,[e_3,[e_1,e_2]]]]\\
=&[e_2,[e_3,[e_1,[e_1,e_2]]]]
\end{align*}
\end{exmp}

\begin{lem}\label{lem:4.5}
Let $(a_n,\dots,a_i,2,\dots,a_1)$ be a nontrivial $n$-tuple such that 
\begin{itemize}
    \item $a_i$ is either $``1"$ or $``3"$;
    \item there are at least three $``2"$ on the right hand side of $a_i$.
\end{itemize}
Then, $(a_n,\dots,a_i,2,\dots,a_1)$ is linearly independent to $(a_n,\dots,2,a_i,\dots,a_1)$.
\end{lem}

\begin{proof}
Similar to proof of Lemma \ref{lem:4.3}, we reduce the $n$-tuple to the following one $(a_i,2,\dots,a_1)$ and we will prove that $(a_i,2,\dots,a_1)$ is linearly independent to $(2,a_i,\dots,a_1)$.

Let $A=[e_{a_{i-2}},[\dots,  [e_{a_{2}},e_{a_{1}}]\dots]]$. Based on Jacobi identity, we have
\begin{align*}
    [e_{a_i},[e_2,A]] = [e_2,[e_{a_i},A]-[A,[e_{a_i},e_2]]
\end{align*}
From the equation above, we just need to check whether $[A,[e_{a_i},e_2]]$ is zero. Clearly, it is zero only under some specific situations listed as follows: 
\begin{itemize}
    \item $e_{a_i}=e_2$
    \item $A=[e_{a_i},e_2]$
\end{itemize}
Since the situations described above are not possible under the assumptions given by the lemma, the two elements, $[e_{a_i},[e_2,A]]$ and $[e_2,[e_{a_i},A]]$, are linearly independent. So are the two $n$-tuples $(\dots,a_i,2,\dots)$ and $(\dots,2,a_i,\dots)$.
\end{proof}

\begin{rmk}\label{rmk:4.6}
In this remark, we want to clarify the exceptions mentioned in Lemma \ref{lem:4.3}. When $e_{a_i}=e_2$, it is obvious that $(\dots,a_i,2,\dots)=(\dots,2,a_i,\dots)$. When $A=[e_{a_i},e_2]$, we have the following example
\begin{align*}
    &[e_2,[e_{a_i},[e_{a_i},e_2]]]\\
    =&[e_{a_i},[e_2,[e_{a_i},e_2]]]+[[e_{a_i},e_2],[e_2,e_{a_i}]]\\
    =&[e_{a_i},[e_2,[e_{a_i},e_2]]]
\end{align*}
In this situation, we find that $e_{a_i}$ and $e_2$ are able to exchange their positions.
\end{rmk}

\begin{rem}\label{rmk:4.7}
Given an $n$-tuple $(a_n,\dots,a_1)$, we highlight the elements $``2"$ in this $n$-tuple as follows
\begin{align*}
    (\dots,2,\dots,2,\dots,2,\dots,2),
\end{align*}
where $``\dots"$ represents sequences of $``1"$ and $``3"$. Lemma \ref{lem:4.5} tells us that the role of $``2"$ is like a board, and the elements $``1"$ and $``3"$ cannot go across the board. If we switch $``1"$ (or $``3"$) with $``2"$ (under the conditions in Lemma \ref{lem:4.5}), then we will get a new linearly independent element. In the meantime, by Lemma \ref{lem:4.3}, switching the positions of two adjacent $``1"$ and $``3"$ does not give us new linearly independent tuples. More precisely, if $e_i$ and $e_j$ ($i,j=1,3$) are between two ``adjacent" $e_2$, then exchanging the positions of $e_i$ and $e_j$ does not give us new element in standard form. Here, two ``adjacent" $e_2$ means that there is no $e_2$ between them. This observation is very important for the calculation of the dimension of root spaces.
\end{rem}

As we discussed above, to calculate the dimension of the root space $\mathfrak{g}_{\lambda}$, it is equivalent to find the largest set of linearly independent elements in standard form. Compared with the combinatorial problem we discussed in \S \ref{sec:3}, we can take $e_1$ and $e_3$ as the red and blue balls respectively, and take $e_2$ as the board. Since $\lambda=n_1 \alpha_1+n_2 \alpha_2 +n_3 \alpha_3$, it is equivalent to consider the combinatorial problem that there are $n_1$ red balls, $n_3$ blue balls and $n_2-1$ boards under some other conditions. The other nontrivial conditions actually come from assumptions we made in Lemma \ref{lem:4.5}.

By Lemma \ref{lem:4.2}, we can assume that the first element in an $n$-tuple is $``2"$, and by Remark \ref{rmk:4.7}, we know that the elements $``2"$ break the $n$-tuple into intervals. The \emph{$i$-th interval} in a given $n$-tuple is the one between the $i$-th $``2"$ and the $(i+1)$-th $``2"$. In Lemma \ref{lem:4.5}, we made the assumption that there are at least two $``2"$ on the right hand side of $a_i$ in the following $n$-tuple
\begin{align*}
(a_n, \dots,2,a_i,\dots,a_1).
\end{align*}
Therefore, elements in the first interval cannot swap with the second $``2"$ in the $n$-tuple. In the next lemma, we make a careful discussion of all possible cases of the numbers of $``1"$ and $``3"$ in the first interval.

\begin{lem}\label{lem:4.8}
Denote by $n'_1$ and $n'_3$ the numbers of $``1"$ and $``3"$ in the first interval of a given nontrivial $n$-tuple. Then, we have $n'_1 \leq a_1$ and $n'_3 \leq a_2$.
\end{lem}

\begin{proof}
Based on the property $(5)$, if $n'_1 \leq a_1$, $n'_3 \leq a_2$, we know the elements corresponding to these tuples $(\underbrace{1,1,\dots,1}_{n'_1},2)$ and $(\underbrace{3,3,\dots,3}_{n'_3},2)$ are trivial elements, where $n'_1 \geq 1+a_1$ and $n'_3 \geq 1+a_2$ respectively. To avoid the trivial cases, we should have $n_1 \leq a_1$ and $n_3 \leq a_2$.
\end{proof}

\subsection{General Case}\label{sec:4.2}
We first give a general calculation for the case that the number of $``1"$ is $i$ and the number of $``3"$ is $j$ in the first interval based on Lemma \ref{lem:4.2} and \ref{lem:4.5}. Then, we sum over all possible cases based on Lemma \ref{lem:4.8}. Note that the elements in standard form corresponding to those $n$-tuples span the root space $\mathfrak{g}_{\lambda}$. Thus, we give an upper bound of the dimension of $\mathfrak{g}_{\lambda}$. 

By Lemma \ref{lem:4.8}, we assume that the number of $``1"$  is $i$ and the number of $``3"$ is $j$ in the first interval, where $i \leq a_1$ and $j \leq a_2$. Then, the rest number of $``1"$ is $n_1-i$, and the rest number of $``3"$ is $n_3-j$. All of these numbers will be arranged into $n_2-2$ intervals. By Lemma \ref{lem:3.1}, we have
\begin{align*}
    {n_1+n_2-i-2 \choose n_2-2}{n_2+n_3-j-2 \choose n_2-2}
\end{align*}
ways in this situation. By summing over all possible situations in the first interval, we have
\begin{align*}
    \sum_{0\leq i \leq a_1, 0\leq j \leq a_2 \atop (i,j)\neq (0,0)}{n_1+n_2-i-2 \choose n_2-2}{n_2+n_3-j-2 \choose n_2-2}
\end{align*}
ways in total.

\subsection{Special Cases}\label{sec:4.3}
In \S \ref{sec:4.2}, we give an upper bound of the dimension of $\mathfrak{g}_{\lambda}$ by summing over all possible $n$-tuples as discussed in \S \ref{sec:4.1}. However, we have counted some trivial $n$-tuples and some $n$-tuples are linearly dependent. In this subsection, we make a careful discussion of these two cases.

\subsubsection{Trivial Tuples}\label{sec:4.3.1} 
In this special case, the first interval only has one element, and then there are at least $a_i-1$ empty interval subsequent to the first one, where $a_i$ depends on the number $``i"$ in the first interval. More precisely, the sequences are given as follows,
\begin{align*}
    (\dots,\underbrace{2,\dots,2}_{a_{1}},1,2) \text{ or } (\dots,\underbrace{2,\dots,2}_{a_2},3,2),
\end{align*} 
where the number of $``2"$ are at least $a_1$ and $a_2$ respectively. Recall the property $(5)$: ad$(e_2)^{1+a_1}(e_1)$ = ad$(e_2)^{1+a_2}(e_3)$ = $0$. Thus, the two sequences we wrote above, namely $(\dots,\underbrace{2,\dots,2}_{a_{1}},1,2)$ and $(\dots,\underbrace{2,\dots,2}_{a_{2}},3,2)$, are all trivial because the corresponding elements equal to zero.

\begin{lem}\label{lem:4.9}
Suppose that $n_2 \geq \max\left \{1+a_1,1+a_2  \right \}$. There are 
\begin{align*}
    {n_1+n_2+n_3-a_1-2 \choose n_2-a_1-1}+{n_1+n_2+n_3-a_2-2 \choose n_2-a_2-1}
\end{align*}
$n$-tuples counted in \S \ref{sec:4.2}, but contributing to trivial cases. All of these $n$-tuples are in the following form 
\begin{align*}
   (\dots,\underbrace{2,\dots,2}_{a_{1}},1,2), \quad (\dots,\underbrace{2,\dots,2}_{a_{2}},3,2).
\end{align*}
\end{lem}

\begin{proof}
Since the number of $``2"$ in the front should be least $a_i$, there will have $a_i$ intervals reduced from the right hand side, where $i$ depends on the number in the first interval. Then the total number of the rest of intervals is $n_2-a_i$. Eliminating the first interval, the number of $e_3$ is $n_1-1+n_3$.

By Lemma \ref{lem:3.1}, there will be
\begin{align*}
&{(n_1-1+n_3)+(n_2-a_1-1) \choose n_2-a_1-1}+ {(n_3-1+n_1)+(n_2-a_2-1) \choose n_2-a_2-1} \\
= & {n_1+n_2+n_3-a_1-2 \choose n_2-a_1-1}+{n_1+n_2+n_3-a_2-2 \choose n_2-a_2-1}
\end{align*}
possible cases in the first special case.
\end{proof}

\subsubsection{Linearly Dependent Tuples}\label{sec:4.3.2}
In Remark \ref{rmk:4.6}, we give an example to explain why we have the extra conditions in Lemma \ref{lem:4.5}. This example tells us that the $n$-tuples we counted in \S \ref{sec:4.1} are not linearly independent. In this subsection, we will find the linearly dependent ones based on Remark \ref{rmk:4.6}.

As shown in Remark \ref{rmk:4.6}, based on the Jacobi Identity, we have
\begin{align*}
    [e_2,[e_{a_{i}},[e_{a_{i}},e_2]]]=[e_{a_{i}},[e_2,[e_{a_{i}},e_2]]].
\end{align*}
Therefore, in the situation above, $e_1$ or $e_3$ can exchange their position with $e_2$, leading to the repetition of many different permutations of elements. For example, $(1,2,3,3,2)$ and $(1,3,2,3,2)$ are linearly equivalent. It is easy to check that the $n$-tuples $(\dots,2,i,i,2)$ and $(\dots,i,2,i,2)$ are linearly equivalent, where $i=1 \text{ or } 3$. However, we count both of them in \S \ref{sec:4.2}. We will calculate the total number of the $n$-tuples in the form of $(\dots,2,i,i,2)$, and subtract from the total number.

\begin{lem} \label{lem:4.10}
Suppose that $n_1,n_2,n_3 \geq 2$ and $a_i > 1$, there are 
\begin{align*}
    {n_1+n_2-4 \choose n_2-2} {n_2+n_3-2 \choose n_2-2}+{n_2+n_3-4 \choose n_2-2} {n_1+n_2-2 \choose n_2-2}
\end{align*}
situations with the form $(\dots,2,i,i,2)$, where $i$ is either $1$ or $3$.

\end{lem}

\begin{proof}
In the proof, we will only discuss the $e_1$ case, as the $e_3$ case shares the same reasoning. Note that there is always an $n$-tuple of the form $(...,2,1,2)$ where there is at least one ``$1$" in the second interval such that it corresponds to a linearly dependent $n$-tuple of the form $(...,2,1,1,2)$. For example, $(3,1,2,1,1,2)$ and $(3,1,1,2,1,2)$ are linearly dependent. Therefore, we can easily conclude that each $n$-tuple of the form $(...,2,1,1,2)$ can be linearly dependent to an $n$-tuple of the form $(...,2,1,2)$, which tells us that the total number of reduced elements is that of elements whose corresponding $n$-tuple is of the form $(...,2,1,1,2)$. Since the answer to $e_3$ case can derived in the same way, the total number to be reduced from second special case is
\begin{align*}
    {n_1+n_2-4 \choose n_2-2} {n_2+n_3-2 \choose n_2-2}+{n_2+n_3-4 \choose n_2-2} {n_1+n_2-2 \choose n_2-2}.
\end{align*}
\end{proof}

\subsection{Main Result}\label{sec:4.4}
We define the following numbers:
\begin{align*}
    &A=\sum_{0\leq i \leq a_1, 0\leq j \leq a_2 \atop (i,j)\neq (0,0)}{n_1+n_2-i-2 \choose n_2-2}{n_2+n_3-j-2 \choose n_2-2}\\
    &B={n_2+n_3-4 \choose n_2-2} {n_1+n_2-2 \choose n_2-2}+{n_1+n_2-4 \choose n_2-2} {n_2+n_3-2 \choose n_2-1}\\
    &C_1={n_1+n_2+n_3-a_1-2 \choose n_2-a_1-1},\quad C_2={n_1+n_2+n_3-a_2-2 \choose n_2-a_2-1}
\end{align*}

With the discussion in \S \ref{sec:4.1}, \ref{sec:4.2} and \ref{sec:4.3}, we have our main theorem.

\begin{thm}\label{thm:4.11}
Suppose that $a_1 \leq a_2$ and $n_1,n_2,n_3 \geq 2$. Let $\lambda = \sum_{i=1}^{3}n_i \alpha_i$. Denote by $\mathfrak{g}_{\lambda}$ the corresponding root space. We have
\begin{align*}
\dim \mathfrak{g}_{\lambda} = \left\{\begin{matrix}
A-B, &n_2 < \min\left \{1+a_1,1+a_2  \right \} \\ 
A-B-C_1, & 1+a_1 \leq n_2 < 1+a_2\\ 
A-B-C_1-C_2, & n_2 \geq \max\left \{1+a_1,1+a_2  \right \}
\end{matrix}\right.
\end{align*}
\end{thm}

\begin{proof}
The number $A$ is the total number in the general case as we discussed in \S \ref{sec:4.2}. The number $B$ corresponds to the special case we discussed in \S \ref{sec:4.3.2}. The number $C_1$ and $C_2$ comes from \S \ref{sec:4.3.1}. The discussion in the above sections give us the theorem.
\end{proof}

\begin{rmk}\label{rmk:4.11.1}
In this remark, we consider the general case that $n_i \geq 0$. Note that if $n_i=0$, then we reduce to a Kac-Moody algebra $\mathfrak{g}'$, of which the associated generalized Cartan matrix is a $2 \times 2$ matrix. Therefore, we assume that $n_i \geq 1$ for $i=1,2,3$.  

Let $n_2 = 1$. by Lemma \ref{lem:4.2}, we know that we can place the only $e_2$ at the beginning of the $n$-tuple. Then, by Lemma \ref{lem:4.5}, we know that if the root space $\mathfrak{g}_{\lambda}$ is nonempty, then $\dim \mathfrak{g_\lambda}=1$.

Now we consider $n_2 \neq 1$ and $n_1=1$. The case of $n_3=1$ can be discussed similarly. Let $\mathfrak{g}':=\mathfrak{g}(A')$ be the Kac-Moody algebra associated to the matrix \begin{align*}
A'=\begin{pmatrix}
2 & -a_2\\
-a_2 & 2
\end{pmatrix},
\end{align*}
and $\mathfrak{g}'$ is a hyperbolic Kac-Moody algebra. Let  $\lambda'=n_2 \alpha_2 + n_3 \alpha_3$. If we forget the only $e_1$, then the dimension of $\mathfrak{g}_{\lambda}$ can be calculated from the dimension of $\mathfrak{g}'_{\lambda'}$ (see \cite{KLL}). Therefore, the only thing we have to consider is the position of the  $``1"$ in the $n$-tuple. With a similar discussion as we did in Lemma \ref{lem:4.8}, we can get a similar formula for $\dim \mathfrak{g}_{\lambda}$ in this case.
\end{rmk}

\subsection{A Special Case of Hyperbolic Kac-Moody Algebras}\label{sec:4.5}
People have made a full classification of hyperbolic Kac-Moody algebras. Some rank-$3$ Hyperbolic Kac-Moody algebras are defined by the generalized Cartan matrix we consider in this paper (see \S \ref{sec:4.1}). In this subsection, we will calculate the root multiplicities of a hyperbolic Kac-Moody algebra as an application of Theorem \ref{thm:4.11}.

Let $a_1=1$ and $a_2=2$. The corresponding Kac-Moody algebra is a hyperbolic Kac-Moody algebra (see \cite{CarChu,Sac}). With the same discussion as in \S \ref{sec:4.2}, the total number of cases in the general case is 
\begin{align*}
     A:=\sum_{0\leq i \leq 1, 0\leq j \leq 2 \atop (i,j)\neq (0,0)}{n_1+n_2-i-2 \choose n_2-2}{n_2+n_3-j-2 \choose n_2-2}.
\end{align*}

For the special case of trivial tuples, we have
\begin{align*}
C_1+C_2:= {n_1+n_2+n_3-4 \choose n_2-3}+{n_1+n_2+n_3-5 \choose n_2-4}
\end{align*}
many situations, in which $n_2 \geq \max\left \{1+a_1,1+a_2  \right \}$ and $n_1, n_3 \geq 1$ based on Lemma \ref{lem:4.10}. 

For the special case about linearly dependent tuples, since $a_1=1$, we will only consider $e_3$ in this example. Therefore, the total number of elements for the special case is
\begin{align*}
B:={n_2+n_3-4 \choose n_2-2} {n_1+n_2-2 \choose n_2-2},
\end{align*}
where $n_1,n_2,n_3 \geq 2$.

\begin{cor}\label{cor:4.13}
Suppose that $n_1,n_2,n_3 \geq 2$. Let $\lambda = \sum_{i=1}^{3}n_i \alpha_i$. Denote by $\mathfrak{g}_{\lambda}$ the corresponding root space. We have
\begin{align*}
\dim \mathfrak{g}_{\lambda} = \left\{\begin{matrix}
A-B, &n_2 < 2 \\ 
A-B-C_1, & n_2 =3\\ 
A-B-C_1-C_2, & n_2 \geq 3
\end{matrix}\right..
\end{align*}
\end{cor}

\section*{Acknowledgements}
We are thankful to great opportunity provided by Pioneer China Program (PCP).

\bigskip
\noindent\small{\textsc{International Department, The Affiliated High School of South China Normal University}\\
No.1 Zhongshan Ave. West, Tianhe District, Guangzhou, Guangdong, China}\\
\emph{E-mail address}:  \texttt{chenbw.bieber2018@gdhfi.com}

\bigskip
\noindent\small{\textsc{International Department, The Affiliated High School of South China Normal University}\\
No.1 Zhongshan Ave. West, Tianhe District, Guangzhou, Guangdong, China}\\
\emph{E-mail address}:  \texttt{luohy.laurie2017@gdhfi.com}

\bigskip
\noindent\small{\textsc{Department of Mathematics, South China University of Technology}\\
381 Wushan Rd, Tianhe Qu, Guangzhou, Guangdong, China}\\
\emph{E-mail address}:  \texttt{hsun71275@scut.edu.cn}

\end{document}